\documentclass{article}
\usepackage{amsmath,amssymb,amsthm,amsxtra,bm,graphicx}
\usepackage[all]{xy}
\usepackage[hypertex]{hyperref}

\newtheorem{thm}{Theorem}[section]

\newtheorem{lem}[thm]{Lemma}
\newtheorem{cor}[thm]{Corollary}
\theoremstyle{definition}
\newtheorem{dfn}[thm]{Definition}
\newtheorem{rem}[thm]{Remark}

\newcommand{\N}{\mathbb{N}}
\newcommand{\Q}{\mathbb{Q}}
\newcommand{\Z}{\mathbb{Z}}
\newcommand{\C}{\mathbb{C}}

\newcommand{\D}{\mathbb{D}}

\newcommand{\oo}{\mathcal{O}}
\newcommand{\M}{\mathrm{M}}

\newcommand{\wtil}[1]{\widetilde{#1}}
\newcommand{\what}[1]{\widehat{#1}}

\newcommand{\geom}{\mathrm{geom}}
\newcommand{\arith}{\mathrm{arith}}
\newcommand{\alg}{\mathrm{alg}}

\newcommand{\ur}{\mathrm{ur}}

\newcommand{\Rep}{\mathrm{Rep}}

\newcommand{\pf}{\mathrm{pf}}

\newcommand{\Sen}{\mathrm{Sen}}
\begin{document}

\title{On Lie algebras arising from $p$-adic representations\\ in the imperfect residue field case}
\author{Shun Ohkubo
\footnote{
Graduate School of Mathematical Sciences, The University of Tokyo, 3-8-1 Komaba Meguro-ku Tokyo 153-8914, Japan. E-mail address: shuno@ms.u-tokyo.ac.jp }
}
\date{\today}

\maketitle

\begin{abstract}
Let $K$ be a complete discrete valuation field of mixed characteristic $(0,p)$ with residue field $k_K$ such that $[k_K:k_K^p]=p^d<\infty$. Let $G_K$ be the absolute Galois group of $K$ and $\rho:G_K\to GL_h(\Q_p)$ a $p$-adic representation. When $k_K$ is perfect, Shankar Sen described the Lie algebra of $\rho(G_K)$ in terms of so-called Sen's operator $\Theta$ for $\rho$. When $k_K$ may not be perfect, Olivier Brinon defined $d+1$ operators $\Theta_0,\dots,\Theta_d$ for $\rho$, which coincides with Sen's operator $\Theta$ in the case of $d=0$. In this paper, we describe the Lie algebra of $\rho(G_K)$ in terms of Brinon's operators $\Theta_0,\dots,\Theta_d$, which is a generalization of Sen's result.
\end{abstract}

\section*{Introduction}
In the series of papers (\cite{AB}, \cite{Bri}, \cite{MSMF}), Fabrizio Andreatta and Olivier Brinon generalized some parts of Fontaine's theory of $p$-adic representations to the relative situation. In \cite[\S~3]{AB} and \cite{Bri}, they extended Sen's theory (\cite[\S\S~1, 2]{Sen2}). Particularly, they defined linear operators on a certain representation, which are analogue of Sen's operator $\Theta$. It is natural to ask whether Sen's theorem on a characterization of the images of $p$-adic representations with respect to $\Theta$ (\cite[Theorem~11]{Sen2}) holds in the relative situation. In this paper, we prove a generalization of this Sen's theorem under the setting in \cite{Bri}, i.e., the base ring is a complete discrete valuation field whose residue field may not be perfect (Theorem~\ref{thm:main}). Sen's original result essentially follows from a certain observation on ramification of $p$-adic Lie extensions. We prove the same ramification result in the imperfect residue field case just by using Borger's generic residual perfection, which preserves ramification (\S~\ref{sec:ram}). For an \'etale fundamental group $\mathcal{G}$ arising in the relative situation, one can expect to deduce some information about the images of $p$-adic representations of $\mathcal{G}$ from our result by restricting $\mathcal{G}$ to its decomposition groups.

\subsection*{Acknowledgement}
The author is supported by Research Fellowships of Japan Society for the Promotion of Science for Young Scientists.

\subsection*{Notation}
Let $p$ be a prime. Let $(K,v_p)$ be a complete valuation field of mixed characteristic $(0,p)$ with $v_p(p)=1$. Denote by $\oo_K$ and $k_K$ the integer ring and the residue field of $K$. In the rest of this paper, we assume that $v_p$ is a discrete valuation. Denote by $\pi_K$ a uniformizer of $K$. Fix an algebraic closure of $K$ and we denote it by $K^{\alg}$ or $\overline{K}$. We denote by $G_K$ (resp. $I_K$) the absolute Galois group of $K$ (resp. the inertia subgroup of $G_K$). For an algebraic extension $L/K$, we endow $L$ with the $p$-adic topology and we denote by $\hat{L}$ the completion of $L$. Particularly, we denote the completion of a maximal unramified extension of $K$ by $K^{\ur}$ and denote $\what{\overline{K}}$ by $\C_p$.

For matrices $X,Y\in \M_h(\C_p)$ and $n\in\Z$, we mean $X\equiv Y\mod{p^n\M_h(\oo_{\C_p})}$ by $X\equiv Y\mod{p^n}$ for simplicity. We define $\exp:2p^n\M_h(\oo_{\C_p})\to 1+2p^nM_h(\oo_{\C_p})$ and $\log:1+2p^n\M_h(\oo_{\C_p})\to 2p^n\M_h(\oo_{\C_p})$ for $n\in\N_{\ge 1}$ as usual and we extend $\log:1+2p^n\Z_p\to 2p^n\Z_p$ to $\log:\Z_p^{\times}\to \Z_p$ by $\log(x):=\log(x^{2(p-1)})/2(p-1)$.

\section{Some ramification theory}\label{sec:ram}
In this section, we generalize a key lemma in \cite{Sen1}, which concerns a ramification of a $p$-adic Lie extension. Throughout this section, we assume that $k_K$ is separably closed.

\begin{lem}[{\cite[Lemma~1]{Sen1}}]\label{lem:Ax}
Let $L/K$ be a Galois extension. Let $x\in\what{L}$ be an element such that for some $n\in\Z$, $(g-1)(x)\in p^n\oo_{\what{L}}$ for all $g\in G_{L/K}$. Then, there exists $y\in K$ such that $x-y\in p^{n-2}\oo_{\what{L}}$.
\end{lem}
\begin{proof}
Since $L$ is dense in $\what{L}$, we may assume $x\in L$. By \cite[Proposition~1]{Ax}, there exists $y\in K$ such that $v_p(x-y)\ge n-(p/(p-1)^2)\ge n-2$, which implies the assertion.
\end{proof}

\begin{dfn}[{cf. \cite[p. 162]{Sen1}}]\label{dfn:Lie}
Let $L/K$ be a Galois extension such that $G=G_{L/K}$ is a $p$-adic Lie group. A Lie filtration of $G$ is a decreasing filtration $\{G_n\}_{n\in\N}$ by open normal subgroups such that for some $n_0\in\N$, $G_{n_0}$ is a $p$-saturated subgroup of $G$ (\cite[Definition~2.1.6]{Laz}) and $G_{n+n_0}=G_{n_0}^{p^n}$ for all $n\in\N$.
\end{dfn}

In the rest of this section, let notation be as in Definition~\ref{dfn:Lie}. We put $K_n:=L^{G_n}$. The following is a key ingredient in the proof of our main theorem.
\begin{lem}[{cf. \cite[Lemma~3]{Sen1}}]\label{lem:key}
Let $\lambda:G_n\to\Q_p$ be a continuous map and $x\in\what{L}$ such that for some $m\in\Z$, we have
\[
\lambda(g)\equiv (g-1)(x)\mod{p^m}\text{ for all }g\in G_n.
\]
Then, there exists a constant $c'_G\in\N$ independent of $n$ such that
\[
\lambda(g)\equiv 0\mod{p^{m-c'_G}}\text{ for all }g\in G_{n}.
\]
\end{lem}
\begin{proof}
Let $K^{\mathrm{g}}$ be the generic residual perfection of $K$ (\cite[1.12]{Bor}). Recall that $K^{\mathrm{g}}$ is a complete discrete valuation field of mixed characteristic $(0,p)$ with perfect residue field and there exists a canonical morphism $K\to K^{\mathrm{g}}$. Moreover, $K$ is algebraically closed in $K^{\mathrm{g}}$ (\cite[Lemma~2.6]{Bor}). Hence, $K$ is algebraically closed in $K^{\mathrm{g},\ur}$, which is the completion of a maximal unramified extension, by Krasner's lemma. Since Galois groups $G$ and $G_n$ are invariant after the base change $K\to K^{\mathrm{g},\ur}$, we may reduce to the perfect residue field case \cite[Lemma~3]{Sen1}.
\end{proof}

\begin{rem}
In the relative situation, the author does not know that an analogue of Lemma~\ref{lem:key} can be reduced to the complete discrete valuation field case since Zariski-Nagata purity theorem may not hold over a non-regular base ring.
\end{rem}

\section{Brinon's generalization of Sen's theory}\label{sec:bri}
In this section, we recall some basic results of Brinon's generalization of Sen's theory (\cite{Bri}). In this section, we assume $[k_K:k_K^p]=p^d<\infty$.

First, we fix some notation. We fix a system of primitive $p^n$-th root $\zeta_{p^n}$ of unity such that $\zeta_{p^{n+1}}^p=\zeta_{p^n}$ for all $n\in\N$. Let $\chi:G_K\to\Z_p^{\times}$ be the cyclotomic character satisfying $g(\zeta_{p^n})=\zeta_{p^n}^{\chi(g)}$. We also fix a lift $t_1,\dots,t_d$ of a $p$-basis of $k_K$ and a system of their $p^n$-th roots $t_1^{p^{-n}},\dots,t_d^{p^{-n}}$ such that $(t_j^{p^{-n-1}})^p=t_j^{p^{-n}}$ for all $n\in\N$ and $1\le j\le d$. We put
\[
K_n:=K(\zeta_{p^n},t_1^{p^{-n}},\dots,t_d^{p^{-n}}),\ K_{\infty}:=\cup_{n\in\N}K_n,
\]
\[
K^{\geom}_n:=K(t_1^{p^{-n}},\dots,t_d^{p^{-n}}),\ K^{\geom}:=\cup_{n\in\N}K^{\geom}_n,
\]
\[
K^{\arith}_n:=K(\zeta_{p^n}),\ K^{\arith}:=\cup_{n\in\N}K^{\arith}_n,
\]
\[
\Gamma^{\geom}_K:=G_{K_{\infty}/K^{\arith}}\hookrightarrow\Gamma_K:=G_{K_{\infty}/K}\twoheadrightarrow\Gamma^{\arith}_K:=G_{K^{\arith}/K},
\]
\[
H_K:=G_{K^{\alg}/K_{\infty}}.
\]
Let $\oo_{\wtil{K}}$ be a Cohen ring of $k_K$ together with an embedding $\oo_{\wtil{K}}\hookrightarrow\oo_K$. We can choose $\oo_{\wtil{K}}$ such that $t_1,\dots,t_d\in\oo_{\wtil{K}}$. For such an $\oo_{\wtil{K}}$, let $\wtil{K}$ be the fraction field of $\oo_{\wtil{K}}$ and we define $\wtil{K}_n$ etc. similarly as above. Then, there exists $\gamma_1,\dots,\gamma_d\in\Gamma_{\wtil{K}}^{\geom}$ such that
\[
\gamma_j(t_i^{p^{-n}})=\zeta_{p^n}^{\delta_{ij}}t_i^{p^{-n}}.
\]
We can choose a section $\Gamma_{\wtil{K}}^{\arith}\to\Gamma_{\wtil{K}}$ of a canonical projection $\Gamma_{\wtil{K}}\to\Gamma_{\wtil{K}}^{\arith}$. By using this section, we may identify $\Gamma_{\wtil{K}}^{\arith}$ as $G_{\wtil{K}_{\infty}/\wtil{K}^{\geom}}$. Then, we have an isomorphism $\iota:\Gamma_{\wtil{K}}\to\Z_p^{\times}\ltimes\Z_p^d$ under which $\gamma_0\in\Gamma_{\wtil{K}}^{\arith}$ (resp. $\gamma_j$) corresponds to $\chi(\gamma_0)$ (resp. $(1,\bm{e}_j)$), where $\bm{e}_j\in\Z_p^d$ is the $j$-th elementary vector. We write $\iota=(\iota_0,\dots,\iota_d)$ and define $\eta_0(g):=\log(\chi(g))$, $\eta_j(g):=\iota_j(g)$ for $1\le j\le d$. Then, $p^{-n}\eta:\Gamma_{\wtil{K}_n}\to\Z_p^{d+1}$ is a chart of the $p$-adic Lie group $\Gamma_{\wtil{K}_n}$ for $n\in\N_{\ge 2}$. We put
\[
\Gamma_{K}^{(j)}:=\{g\in\Gamma_K;\eta_0(g)=\stackrel{\stackrel{j}{\check{}}}{\dots}=\eta_d(g)=0\}
\]
for $0\le j\le d$. Since $\Gamma_K$ is an open subgroup of $\Gamma_{\wtil{K}}$, $\eta_j:\Gamma_K^{(j)}\to\Z_p$ is locally isomorphic.

A $\C_p$-representation of $G_K$ is a finite dimensional $\C_p$-vector space $V$ with continuous semi-linear $G_K$-action. Denote by $\Rep_{\C_p}G_K$ the category of $\C_p$-representations of $G_K$. Similarly, we define categories $\Rep_{\what{K}_{\infty}}\Gamma_K$, $\Rep_{K_{\infty}}\Gamma_K$ and $\Rep_{\Q_p}G_K$, where $\what{K}_{\infty}$ is the $p$-adic completion of $K_{\infty}$. Then, the functors
\[
\Rep_{\what{K}_{\infty}}\Gamma_K\to\Rep_{\C_p}G_K;V\mapsto V\otimes_{\Q_p}\C_p,
\]
\[
\Rep_{K_{\infty}}\Gamma_K\to\Rep_{\what{K}_{\infty}};V\mapsto V\otimes_{\Q_p}\what{K}_{\infty}
\]
are equivalences of categories (\cite[Theorem~4]{Bri}). Thus, we obtain an equivalence of categories $\D_{\Sen}:\Rep_{\C_p}G_K\to\Rep_{K_{\infty}}\Gamma_K$. Note that $\D_{\Sen}(V)$ is defined over $K_n$ for all sufficiently large $n\in\N$, i.e., there exists a finite dimensional $K_n$-subspace $\D_{\Sen,n}(V)$ of $\D_{\Sen}(V)$ stable by $\Gamma_{K_n}$ such that a canonical map $K_{\infty}\otimes_{K_n}\D_{\Sen,n}(V)\to\D_{\Sen}(V)$ is an isomorphism. For $\gamma\in\Gamma_K$ and $x\in\D_{\Sen}(V)$, the series $\log(\gamma)(x):=\sum_{n\ge 1}(-1)^{n-1}(\gamma-1)^nx/n$ converges and it defines a $K_{\infty}$-endomorphism $\log(\gamma)$ on $\D_{\Sen}(V)$. For $0\le j\le d$, we put
\[
\Theta_j:=\frac{\log(\gamma^{(j)})}{\eta_j(\gamma^{(j)})},
\]
where $\gamma^{(j)}\in\Gamma_K^{(j)}\setminus\{1\}$ is sufficiently close to $1$. Note that $\Theta_0,\dots,\Theta_d$ are independent of the choices of the $\gamma^{(j)}$'s. Recall that we have the relations $[\Theta_0,\Theta_j]=\Theta_j$ for $1\le j\le d$ and $[\Theta_i,\Theta_j]=0$ for $1\le i,j\le d$. Also, note that we can recover the action of $\Gamma_K$ on $\D_{\Sen}(V)$ by the formula
\[
\gamma^{(j)}=\exp(\eta_j(\gamma^{(j)})\Theta_j)\text{ for }0\le j\le d,
\]
where $\gamma^{(j)}\in\Gamma_K^{(j)}\setminus\{1\}$ is sufficiently close to $1$. Note that $\D_{\Sen,n}(V)$ is stable by $\Theta_j$ for all sufficiently large $n$.

Finally, we note that Brinon's operators $\Theta_0,\dots,\Theta_d$ are compatible with a certain base change as follows. For $1\le j\le d$, we define $K^{(j)}$ as the completion of $\cup_{n\in\N}K(t_1^{p^{-n}},\dots,t_{j-1}^{p^{-n}},t_{j+1}^{p^{-n}},\dots,t_d^{p^{-n}})$. Then, $K^{(j)}$ is a complete discrete valuation field. In the following, we regard $G_{K^{(j)}}$ as a closed subgroup of $G_K$. We choose $t_j$ as a lift of a $p$-basis of the residue field $k_{K^{(j)}}$. For $V\in\Rep_{\C_p}G_K$, there exists a canonical isomorphism $K^{(j)}_{\infty}\otimes_{K_{\infty}}\D_{\Sen}(V)\to\D_{\Sen}(V|_{K^{(j)}})$ and $\Theta_0,\Theta_j$ are the associated Brinon's operators to $V|_{K^{(j)}}$.

Similarly, we define $K^{\pf}$ as the completion of $K^{\geom}$. Then, $K^{\pf}$ is a complete discrete valuation field with perfect residue field. There exists a canonical isomorphism $K^{\pf}_{\infty}\otimes_{K_{\infty}}\D_{\Sen}(V)\to \D_{\Sen}(V|_{K^{\pf}})$ and $\Theta_0$ is the associated Sen's operator to $V|_{G_{K^{\pf}}}$.


\section{Main Theorem}
Let notation be as in \S~\ref{sec:bri}. Let $\rho:G_K\to GL_h(\Q_p)$ be a $p$-adic representation of $G_K$. We identify $V\otimes_{\Q_p}\C_p$ as $\D_{\Sen}(V)\otimes_{K_{\infty}}\C_p$. Let $\Theta_0,\dots,\Theta_d$ be Brinon's operators for the $\C_p$-representation $V\otimes_{\Q_p}\C_p$ of $G_K$. We denote the scalar extension of $\Theta_0,\dots,\dots,\Theta_d$ with respect to $K_{\infty}\to\C_p$ by $\Theta_0,\dots,\Theta_d$ again. Let $\mathfrak{g}$ be the Lie algebra of $\rho(I_K)$. The following is our main theorem in this paper.
\begin{thm}[{cf. \cite[Theorem~11]{Sen2}}]\label{thm:main}
The Lie algebra $\mathfrak{g}$ is the smallest $\Q_p$-subspace $S$ of $\mathrm{End}_{\Q_p}V$ such that $S\otimes_{\Q_p}\C_p$ contains $\Theta_0,\dots,\Theta_d$.
\end{thm}

As a special case (precisely, the case of $\Theta_0=\dots=\Theta_d=0$) of Theorem~\ref{thm:main}, we can reprove the following theorem (\cite[Theorem~2.1]{Ohk}).
\begin{cor}[{cf. \cite[Corollary in (3.2)]{Sen2}}]
If $V$ is $\C_p$-admissible, then $I_K$ acts on $V$ via a finite quotient.
\end{cor}

We prove Theorem~\ref{thm:main} in the rest of this section. By replacing $K$ by $K^{\ur}$, we may assume that $k_K$ is separably closed. In particular, we may use the results in \S~\ref{sec:ram}. We first fix some notation.  We fix a $\Q_p$-basis $e_1,\dots,e_h$ of $V$ and let $U_{\bullet}:G_K\to GL_h(\Q_p)$ be the continuous group homomorphism defined by $g(e_1,\dots,e_h)=(e_1,\dots,e_h)U_g$ for $g\in G_K$. We fix a $K_{\infty}$-basis $e'_1,\dots,e'_h$ of $\D_{\Sen}(V)$ and let $U'_{\bullet}:G_K\to GL_h(\C_p)$ be the continuous $1$-cocycle defined by $g(e'_1,\dots,e'_h)=(e'_1,\dots,e'_h)U'_g$ for $g\in G_K$. Let $M\in GL_h(\C_p)$ be the matrix transforming the $e'_i$'s into the $e_i$'s, i.e., $(e_1,\dots,e_h)=(e'_1,\dots,e'_h)M$. By multiplying $p^m$ to the $e_i$'s, we assume $M\in \mathrm{M}_h(\oo_{\C_p})$. Denote the matrix presentations of $\Theta_0,\dots,\Theta_d$ with respect to $e'_1,\dots,e'_h$ by $\Theta_0,\dots,\Theta_d\in \mathrm{M}_h(K_{\infty})$ again. Put $A_j:=M^{-1}\Theta_j M$, which is a matrix presentation of $\Theta_j$ with respect to $e_1,\dots,e_d$.
\begin{dfn}
Let $F$ be a complete valuation field. An $F$-linear form $f$ on $\M_h(F)$ is an $F$-linear map $f:\M_h(F)\to F$. An $F$-linear form $f$ is integral if $f(\M_h(\oo_F))\subset \oo_F$. Let $f$ be an $F$-linear form on $\M_h(F)$ and $F\to F'$ an extension of complete valuation fields. By extending scalar, we regard $f$ as an $F'$-linear form on $\M_h(F')$. Note that if $f$ is integral, then so is its extension.
\end{dfn}

By duality, Theorem~\ref{thm:main} is equivalent to the following:
\begin{thm}[{cf. \cite[Theorem~1']{Sen1}}]\label{thm:equiv}
For any integral $\Q_p$-linear form $f$ on $\M_h(\Q_p)$,
\[
f(A_0)=\dots=f(A_d)=0\Leftrightarrow f(\mathfrak{g})=0.
\]
Furthermore, for any open subgroup $\mathcal{U}$ of $G_K$, the latter condition is equivalent to say that $f(\log(U_g))=0$ for all $g\in \mathcal{U}$ since $\mathfrak{g}$ is generated by $\{\log(U_g);g\in\mathcal{U}\}$ as a $\Q_p$-vector space.
\end{thm}

We choose $m_0\in\N_{\ge 2}$ sufficiently large such that
\begin{enumerate}
\item $\D_{\Sen}(V)$ is defined over $K_{m_0}$;
\item $\Theta_0,\dots,\Theta_d\in \M_h(K_{m_0})$;
\item $\Gamma_K$ contains $\Gamma_{\wtil{K}_{m_0}}$.
\end{enumerate}
We choose $c\in\N$ such that $p^c\Theta_0,\dots,p^c\Theta_d\in \M_h(\oo_{K_{m_0}})$. For $m\ge m_0$, we define $G_m:=\{g\in G_{K_{m+c}};U_g\in 1+p^m\mathrm{M}_h(\Z_p)\}$ and $G_{\infty}:=\cap_{m\ge 1}G_m=\{g\in H_K;U_g=1\}$. We also define $\check{G}_:=G_m/G_{\infty}$ for $m\ge m_0$. Then, $\check{G}_{m_0}$ is a $p$-adic Lie group and $\{\check{G}_m\}_{m\ge m_0}$ is a Lie filtration of it. We may regard $U$ as a $1$-cocycle on $\check{G}_m$. Note that we have $\{\log(U_g);g\in G_{m_0}\}=\{\log(U_g);g\in \check{G}_{m_0}\}$.

Before proving Theorem~\ref{thm:equiv}, we gather some basic lemmas.
\begin{lem}\label{lem:Sen}
For $m\ge m_0$ and $g\in\Gamma_{K_{m+c}}$, we have
\[
U'_g\equiv 1+\sum_{0\le j\le d}\eta_j(g)\Theta_j\mod{p^{2m-1}}.
\]
\end{lem}
\begin{proof}
Since $\Gamma_{K_{m+c}}\cong\Gamma_{\wtil{K}_{m+c}}\cong (1+p^{m+c}\Z_p)\ltimes p^{m+c}\Z_p^d$, there exists $g_j\in\Gamma_{K_{m+c}}^{(j)}$ for $0\le j\le d$ such that $g=g_dg_{d-1}\dots g_0$. Since $U'_{\bullet}\in \M_h(K_{m_0})$, we have
\begin{align*}
U'_g=U'_{g_d}\dots U'_{g_0}&\equiv\exp(\eta_d(g_d)\Theta_d)\dotsb\exp(\eta_0(g_0)\Theta_0)\equiv (1+\eta_d(g_d)\Theta_d)\dotsb (1+\eta_0(g_0)\Theta_0)\\
&\equiv 1+\eta_d(g_d)\Theta_d+\dots+\eta_0(g_0)\Theta_0\equiv 1+\eta_d(g)\Theta_d+\dots+\eta_0(g)\Theta_0\mod{p^{2m-1}}.
\end{align*}
Here, we use the congruence $\exp(A)\equiv 1+A\mod{p^{2m-1}}$ for $A\in \M_h(\oo_{\C_p})$.
\end{proof}

\begin{cor}\label{cor:Sen}
For $m\ge m_0$ and $g\in\Gamma_{K_{m+c}}$, we have
\[
U'_g\equiv 0\mod{p^{m}}.
\]
\end{cor}
\begin{proof}
It follows from the fact that for $g\in\Gamma_{K_{m+c}}$, we have $\eta_j(g)\Theta_j\equiv 0\mod{p^m}$.
\end{proof}

We prove key congruences (\ref{eq:eq3}) and (\ref{eq:eq5}) in the following. By definition, we have
\begin{equation}\label{eq:eq1}
MU_g=U'_gg(M)
\end{equation}
for all $g\in G_K$. Hence, we have $M=g(M)$ for all $g\in G_{\infty}$, i.e., $M\in \M_h(\C_p^{G_{\infty}})$. In the following, we may regard (\ref{eq:eq1}) as an equation for $g\in\check{G}_{m_0}$. Let $m\ge m_0$. By (\ref{eq:eq1}) and Corollary~\ref{cor:Sen}, we have $g(M)\equiv M\mod{p^m}$ for $g\in\check{G}_m$. By Lemma~\ref{lem:Ax}, there exists $M_m\in \M_h(\C_p^{G_m})$ such that
\[
M_m\equiv M\mod{p^{m-2}}.
\]
For a while, let $g\in\check{G}_m$. We have congruences
\[
U_g\equiv 1+\log(U_g)\mod{p^{2m-1}},
\]
\[
U'_g\equiv 1+\sum_{0\le j\le d}\eta_j(g)\Theta_j\mod{p^{2m-1}},
\]
where the last congruence follows from Lemma~\ref{lem:Sen}. By substituting these congruences in (\ref{eq:eq1}), we obtain a congruence
\[
M+M\log(U_g)\equiv g(M)+\sum_{0\le j\le d}\eta_j(g)\Theta_jg(M)\mod{p^{2m-1}}.
\]
Since $\log(U_g)\equiv0\mod{p^m}$ and $\eta_j(g)\equiv 0\mod{p^{m+c}}$ for $0\le j\le d$, we have
\begin{equation}\label{eq:eq2}
M+M_m\log(U_g)\equiv g(M)+\sum_{0\le j\le d}\eta_j(g)\Theta_jM_m\mod{p^{2m-2}}.
\end{equation}

We choose sufficiently large $r\in\N_{\ge 2}$ such that $p^{r-2}\det(M)^{-1}\in\oo_{\C_p}$. In the following, we assume $m>r$. Since $v_p(\det(M))=v_p(\det(M_m))$ by assumption, we have $p^{r-2}M_m^{-1}\in \M_h(\oo_{\C_p})$. By multiplying (\ref{eq:eq2}) on the left by $p^{r-2}M_m^{-1}$ and dividing $p^{r-2}$, we obtain a congruence
\[
C_m+\log(U_g)\equiv g(C_m)+\sum_{0\le j\le d}\eta_j(g)M_m^{-1}\Theta_jM_m\mod{p^{2m-r}},
\]
where $C_m:=M_m^{-1}M\equiv 1\mod{p^{m-r}}$. We can rewrite the above congruence as
\begin{equation}\label{eq:eq3}
(g-1)C_m\equiv \log(U_g)-\sum_{0\le j\le d}\eta_j(g)A_{j,m}\mod{p^{2m-r}},
\end{equation}
where $A_{j,m}:=M_m^{-1}\Theta_jM_m\in \M_h(\C_p^{G_m})$. Since $M_m\to M$ as $m\to\infty$, $A_{j,m}$ converges to $A_j$ as $m\to\infty$. Apply an integral $\Q_p$-linear form $f$ on $\M_h(\Q_p)$ to (\ref{eq:eq3}), we obtain
\begin{equation}\label{eq:eq4}
(g-1)f(C_m)\equiv f(\log(U_g))-\sum_{0\le j\le d}\eta_j(g)f(A_{j,m})\mod{p^{2m-r}}.
\end{equation}
By $\eta_j(g)\equiv 0\mod{p^{m+c}}$ and
\begin{equation}\label{eq:cong}
A_{j,m}=M_m^{-1}M\cdot M^{-1}\Theta_j M \cdot M^{-1}M_m\equiv A_j\mod{p^{m-2r+2-c}},
\end{equation}
we have
\begin{equation}\label{eq:eq5}
(g-1)f(C_m)\equiv f(\log(U_g))-\sum_{0\le j\le d}\eta_j(g)f(A_{j})\mod{p^{2m-2r+2}}.
\end{equation}

We prove Theorem~\ref{thm:equiv}. First, we assume $f(A_0)=\dots=f(A_d)=0$. By applying Lemma~\ref{lem:key} to each entry of (\ref{eq:eq5}), we have
\[
f(\log(U_g))\equiv 0\mod{p^{2m-2r-2-c'}}\text{for all }g\in\check{G}_m,
\]
where $c'$ is a constant for $G=\check{G}_{m_0}$ in Lemma~\ref{lem:key}. Since $g^{p^{m-m_0}}\in\check{G}_m$ and $f(\log(U_{g^{p^{m-m_0}}}))=p^{m-m_0}f(\log(U_g))$ for $g\in\check{G}_{m_0}$, we have
\[
f(\log(U_g))\equiv 0\mod{p^{m+m_0-2r+2+c'}}\text{for all }g\in\check{G}_{m_0}
\]
Since this congruence holds for all sufficiently large $m\in\N$, we have $f(\log(U_g))=0$ for all $g\in\check{G}_{m_0}$ by passing $m\to\infty$.

Finally, we prove $f(A_0)=\dots=f(A_d)=0$ by assuming $f(\log(U_g))=0$ for all $g\in\check{G}_{m_0}$. First, note that the assumption implies $f(\mathfrak{g})=0$. We define $\mathfrak{g}^{\pf}$ as the Lie algebra associated to $\rho(I_{K^{\pf}})$. Since $\mathfrak{g}^{\pf}\subset\mathfrak{g}$ by definition, we have $f(\mathfrak{g}^{\pf})=0$ by assumption. By applying Sen's result (\cite[Theorem~1']{Sen1}) to $V|_{K_{\pf}}$, we have $f(A_0)=0$. We suppose that $f(A_j)\neq 0$ for some $1\le j\le d$ and we deduce a contradiction. By replacing $V$ by $V|_{K^{(j)}}$, we may assume $d=j=1$. By (\ref{eq:cong}) and (\ref{eq:eq5}), we have
\[
(g-1)f(C_m)\equiv-\eta_1(g)f(A_{1,m})\mod{p^{2m-2r+2}}
\]
for all $g\in\check{G}_m$. We fix $s\in\N$ such that $p^sf(A_1)^{-1}\in\oo_{\C_p}$. Since $A_{1,m}\to A_1$ as $m\to\infty$, we have $f(A_{1,m})\neq 0$ and $p^sf(A_{1,m})^{-1}\in\oo_{\C_p}$ for all sufficiently large $m$. Hence, we have
\[
(g-1)\left(\frac{f(C_m)}{f(A_{1,m})}\right)\equiv-\eta_1(g)\mod{p^{2m-2r+2-s}}.
\]
By Lemma~\ref{lem:key}, there exists $c'\in\N$ independent of $m$ such that $\eta_1(g)\equiv 0\mod{p^{2m-2r+2-s-c'}}$ for all $g\in\check{G}_m$. Let $g\in\check{G}_{m_0}$. Since $g^{p^{m-m_0}}\in\check{G}_{m}$, we have $\eta_1(g)\equiv 0\mod{p^{m+m_0-2r+2-s-c'}}$ for all sufficiently large $m$. By passing $m\to\infty$, we have $\eta_1(g)=0$ for all $g\in\check{G}_{m_0}$. Since $G_{m_0}$ is open in $G_K$, $\eta_1\not\equiv 0$ on $\check{G}_{m_0}$, which is a contradiction. Thus, we finish the proof of Theorem~\ref{thm:equiv}.







\end{document}